\documentclass[a4paper, 12pt]{article}

\usepackage{amssymb,amsmath,amsthm,sectsty,url,mathrsfs,graphicx}
\usepackage[letterpaper,hmargin=1.0in,vmargin=1.0in]{geometry}
\usepackage[colorlinks,linkcolor=black,citecolor=black,filecolor=black,urlcolor=black]{hyperref}
\usepackage{color}

\usepackage[british]{babel}

\usepackage{xspace}
\usepackage{calc}
\usepackage{euscript}
\usepackage{amsfonts}
\usepackage{enumerate}

\usepackage[textsize=tiny]{todonotes}

\usepackage{a4wide}
\parskip \medskipamount

\sectionfont{\large}
\subsectionfont{\normalsize}
\numberwithin{equation}{section}

\newtheorem{theorem}{Theorem}
\newtheorem{lemma}[theorem]{Lemma}

\newtheorem{corollary}[theorem]{Corollary}

\newtheorem{remark}[theorem]{Remark}

\newtheorem{question}[theorem]{Question}

\def\cE{\mathcal{E}}

\def\P{\mathbb{P}}

\newcommand{\E}{{\mathbb{E}}}

\newcommand{ \mix}{ t_{\mathrm{mix}} }

\newcommand{ \cL}{ \mathcal L }

\newcommand{\la}{\lambda}

\DeclareMathSymbol{\leqslant}{\mathalpha}{AMSa}{"36} 
\DeclareMathSymbol{\geqslant}{\mathalpha}{AMSa}{"3E} 
\DeclareMathSymbol{\eset}{\mathalpha}{AMSb}{"3F}     
\renewcommand{\le}{\;\leqslant\;}                   
\renewcommand{\ge}{\;\geqslant\;}                   

\renewcommand{\epsilon}{\varepsilon}

\newcommand{\N}{\mathbb N}
\newcommand{\R}{\mathbb R}
\newcommand{\Z}{\mathbb Z}

\usepackage[normalem]{ulem}

\begin{document}

\title{A version of Aldous' spectral-gap conjecture for the zero range process}
\author{Jonathan Hermon
\thanks{
University of Cambridge, Cambridge, UK. E-mail: {\tt jonathan.hermon@statslab.cam.ac.uk}. Financial support by
the EPSRC grant EP/L018896/1.}
\and Justin Salez
\thanks{ Universit\'e Paris Diderot, 75205 Paris CEDEX 13, France. E-mail: {\tt justin.salez@lpsm.paris}.} 
}
\date{}
\maketitle

\begin{abstract}
We show that the spectral gap of a general zero range process can be controlled in terms of the spectral gap for a single particle. This is in the spirit of Aldous' famous spectral-gap conjecture for the interchange process, now resolved by Caputo et al. Our main inequality decouples the role of the geometry (defined by the jump matrix) from that of the kinetics (specified by the exit rates). Among other consequences, the various spectral gap estimates that were so far only available on the complete graph or the $d$-dimensional torus now extend effortlessly to arbitrary geometries. As an illustration, we determine the exact order of magnitude of the spectral gap of the rate-one zero-range process on any regular graph and, more generally, for any doubly stochastic jump matrix.

\end{abstract}

\section{Introduction}
Introduced by Spitzer in 1970 \cite{Spitzer}, the zero-range process (\textbf{ZRP}) is one of the most widely studied models of interacting particles \cite{liggettbook2,liggettbook1}. It describes the evolution of a fixed number $m\ge 1$ of indistinguishable particles diffusing on a finite set $V$ of sites. The state space is
\begin{eqnarray}
\Omega & := & \left\{\eta\in\Z_+^V\colon \sum_{x\in V}\eta(x)=m\right\},
\end{eqnarray}
and the dynamics are specified by the following two ingredients:
\begin{itemize}
\item[$\bullet$] a collection of positive numbers ${\mathbf r}=\left(r(x,k)\colon x\in V, k\in \N\right)$ defining the \emph{kinetics}:  $r(x,k)$ is the rate at which particles are expelled from $x$ if $x$ is occupied by $k$ particles;
\item[$\bullet$]an irreducible stochastic matrix $P$ on $V$ defining the \emph{geometry}:  $P(x,y)$ indicates the probability that a particle expelled from $x$ goes to $y$. 
\end{itemize} 
 Formally, the Markov generator $\cL$ of the \textbf{ZRP} acts on an observable $f\colon \Omega\to\R$ as follows:
\begin{eqnarray}
\label{e:gen}
(\cL f)(\eta) & = & \sum_{(x,y)\in V^2}r(x,\eta(x))P(x,y)\left(f(\eta+\delta_y-\delta_x)-f(\eta)\right),
\end{eqnarray}
where $(\delta_x)_{x\in V}$ denotes the canonical basis of $\R^V$, and with the convention that $r(x,0)=0$ for all $x\in V$. This generator is clearly irreducible, and the stationary law $\mu$ can be checked to admit the following explicit product form:
\begin{eqnarray}
\label{statio}
\mu(\eta) & = & \frac{1}{Z}\prod_{x\in V}\prod_{k=1}^{\eta(x)}\frac{\pi(x)}{r(x,k)},
\end{eqnarray}
where $\pi$ is the invariant law of $P$, and $Z$ a normalization factor. The basic theory ensures that the system will converge to its stationary law, and the present paper is concerned with the general problem of quantifying the rate at which this convergence occurs. 

To answer this question, let us recall that the Dirichlet form associated with our process is
\begin{eqnarray}
\label{def:dir}
\cE_{\textsc{zrp}(P,{\bf r},m)}\left(f,g\right) & := & -\langle f,\cL g\rangle_\mu,
\end{eqnarray}
where $\langle f,g\rangle_\mu:=\sum_{\eta\in\Omega}\mu(\eta)f(\eta)g(\eta)$ denotes the usual inner-product in $\ell^2(\Omega,\mu)$. Recall also that the associated \emph{Poincar\'e constant}  is defined by
\begin{eqnarray}
\label{e:1.5}
\lambda\left(\textsc{zrp}(P,{\bf r},m)\right) & := & \min\left\{\frac{\cE_{\textsc{zrp}(P,{\bf r},m)}\left(f,f\right)}{\textrm{Var}_\mu(f)}\right\},
\end{eqnarray}
where the minimum runs over all non-constant observables, and $\textrm{Var}_\mu(f)$ denotes the variance of $f$ under the stationary law. The question of estimating $\lambda\left(\textsc{zrp}(P,{\bf r},m)\right)$ has received quite some attention. The most standard setting is that where $P$ is the transition matrix of simple random walk on some regular graph $G$, with the following two choices of kinetics:
\begin{enumerate}
\item the rate-one case, where $r(x,k)=1$ for all $x\in V$ and all $k \ge 1$; 
\item the case of homogeneous Lipschitz rates increasing at infinity, see (\ref{lip1})-(\ref{lip2}) below.
\end{enumerate}
In the first case, the exact order of magnitude of the Poincar\'e constant was determined by Morris \cite{MorrisZRP} on the complete graph $G=K_n$ and the torus $G=(\Z/n\Z)^d$. In the second case, a uniform lower-bound  was established by  Landim,  Sethuraman and Varadhan on the torus \cite{MR1415232} and by Caputo on the complete graph \cite{MR2073330}.
The total absence of results for other geometries might seem surprising, especially compared to the cases of the exclusion and interchange processes \cite{Caputo,Olive,MR3069380}. The fundamental reason is the lack of an analogue of Aldous' famous spectral-gap conjecture -- now resolved by Caputo, Liggett and Richthammer \cite{Caputo} -- to reduce the understanding of the whole system to that of a single particle.

In this note, we provide a version of this missing `many-to-one' reduction by exhibiting a simple connection between the Dirichlet form  $\cE_{\textsc{zrp}(P,{\bf r},m)}$ of the zero-range process and the Dirichlet form $\cE_P$ of its jump matrix $P$, defined on $\ell^2(V,\pi)$ by
\begin{eqnarray}
\cE_{P}(\phi,\psi) & := & \langle \phi,(I-P)\psi\rangle_\pi,
\end{eqnarray}
with $\langle \phi,\psi\rangle_\pi:=\sum_{x\in V}\pi(x)\phi(x)\psi(x)$. Among other consequences, we transfer all the results mentioned above to arbitrary regular graphs and, more generally, to \emph{any} doubly stochastic jump matrix $P$, at the optimal cost of a multiplication by the Poincar\'e constant of $P$:
\begin{eqnarray}
\lambda(P) & := & \min\left\{\frac{\cE_P\left(\phi,\phi\right)}{\textrm{Var}_\pi(\phi)}\right\}.
\end{eqnarray}
In particular, we explicitly determine the exact order of magnitude of the Poincar\'e constant of the rate-one $\textbf{ZRP}$ on any regular graph.

\begin{remark}[Reversibility]
If the jump matrix $P$ satisfies the reversibility property
\begin{eqnarray}
\label{reversible}
\forall (x,y)\in V^2,\qquad \pi(x)P(x,y) & = & P(y,x)\pi(x),
\end{eqnarray}
then both $P$ and $\cL$ are self-adjoint operators on $\ell^2(V,\pi)$ and $\ell^2(\Omega,\mu)$ respectively. Consequently,  the Poincar\'e constants $\lambda(P)$ and $\lambda\left(\textsc{zrp}(P,{\bf r},m)\right)$ coincide with the more classical spectral gap, that is, the smallest non-zero eigenvalue of the positive semi-definite operators $I-P$ and $-\cL$ respectively. However, we emphasize that none of our results require (\ref{reversible}).
\end{remark}

\section{Results}

One of the most powerful techniques to analyze a complicated Markov chain consists in comparing its Dirichlet form with that of a better understood chain having the same state space and stationary distribution   \cite{comparison}. In the case of the \textbf{ZRP}, we will show that this comparison can be performed directly at the  level of the jump matrix $P$, without any loss. 

More precisely, let $Q$ be another stochastic matrix with stationary law $\pi$ on $V$, and consider the new \textbf{ZRP} obtained by replacing the jump matrix $P$ with $Q$, while \emph{keeping the rates and the number of particles unchanged}. This modification preserves both the state space $\Omega$ and the stationary law $\mu$, and is therefore eligible for an application of the comparison method. Specifically, we seek an explicit constant $\kappa>0$ (as large as possible) such that 
\begin{eqnarray}
\label{comp:zrp}
\forall f\in \R^\Omega, \qquad \cE_{\textsc{zrp}(P,{\bf r},m)}(f,f) & \ge & \kappa\, {\cE}_{\textsc{zrp}(Q,{\bf r},m)}(f,f).
\end{eqnarray}
This may not seem easy to achieve at all, given the complexity of the particle system. Yet, our main result asserts that
(\ref{comp:zrp}) is in fact \emph{precisely} equivalent to the simpler inequality
\begin{eqnarray}
\label{comp:jump}
\forall \phi\in \R^V, \qquad \cE_{P}(\phi,\phi) & \ge & \kappa\, {\cE}_{Q}(\phi,\phi).
\end{eqnarray}
Thus, comparing two \textbf{ZRP} boils down to comparing their jump matrices.

\begin{theorem}\label{thm:main}The inequality (\ref{comp:zrp}) holds if and only if  (\ref{comp:jump}) holds. In other words,
\begin{eqnarray}
\label{main}
\min_f\left\{\frac{\cE_{\textsc{zrp}(P,{\bf r},m)}(f,f)}{\cE_{\textsc{zrp}(Q,{\bf r},m)}(f,f)}\right\} & = & \min_\phi\left\{\frac{\cE_{P}(\phi,\phi)}{\cE_{Q}(\phi,\phi)}\right\},
\end{eqnarray}
where the $\min$ on the left (resp.\ right) is over all non-constant functions on $\Omega$ (resp.\ $V$).
\end{theorem}

We emphasize that the result is valid for any number of particles and any choice of the underlying rates, as long as the same are used in both processes. A particularly interesting choice for the matrix $Q$ is $\Pi$, the matrix with all rows equal to $\pi$:
\begin{eqnarray}
\Pi(x,y) & := & \pi(y).
\end{eqnarray}
In this case, we are comparing our \textbf{ZRP} to its \emph{mean-field} version, where all jump destinations are sampled afresh from the stationary law $\pi$. In particular, we have
\begin{eqnarray}
\cE_{\Pi}(\phi,\phi) & = & \frac 12 \sum_{x,y}\pi(x)\pi(y)(\phi(x)-\phi(y))^2 \ = \textrm{Var}_\pi(\phi)
\end{eqnarray}
so the right-hand side of (\ref{main}) is  nothing but the Poincar\'e constant $\lambda(P)$ of the transition matrix $P$. We thus obtain the following inequality, which completely \emph{decouples} the contribution of the geometry from that of the kinetics:
\begin{corollary}[Comparison to the Mean-Field version]\label{co:mf} For all observables $f\colon\Omega\to\R$, 
\begin{eqnarray}
\label{comp:mf}
{\cE_{\textsc{zrp}(P,{\bf r},m)}(f,f)} & \ge & \lambda(P)\,\cE_{\textsc{zrp}(\Pi,{\bf r},m)}(f,f).
\end{eqnarray}
Moreover, the constant $\lambda(P)$ is optimal, i.e.\ there exists a non-constant observable $f\colon\Omega\to\R$ such that the equality holds in (\ref{comp:mf}).
\end{corollary}

If $P$ is the transition matrix of simple random walk on a regular $n$-vertex graph, or more generally, if $P$ is a doubly stochastic $n\times n$ matrix, then $\Pi$ is simply the  matrix with all entries equal to $\frac 1n$,  which we will denote by $K_n$. There are many results available for $\textsc{zrp}(K_n,{\mathbf{r}},m)$, making our comparison quite fruitful. For example, a uniform lower-bound on the Poincar\'e constant was established in \cite{MR2200172} when the rates are uniformly increasing:
\begin{eqnarray}
\lambda\left({\textsc{zrp}(K_n,{\bf r},m)}\right) & \ge & \inf_{x,k}\left\{r(x,k+1)-r(x,k)\right\}.
\end{eqnarray}
 By virtue of Corollary \ref{co:mf}, we immediately obtain the following considerable extension.
\begin{corollary}[Uniformly increasing rates]For any doubly stochastic matrix $P$,
\begin{eqnarray}
\lambda\left(\textsc{zrp}(P,{\bf r},m)\right) & \ge & \lambda(P)\,\inf_{x,k}\left\{r(x,k+1)-r(x,k)\right\}.
\end{eqnarray}
\end{corollary}
This inequality is sharp, as can be seen by considering the case of independent walkers ($r(x,k)=k$ for all $x\in V,k\ge 1$). 

As a second example, let us consider the extensively studied case of homogeneous Lipschitz rates increasing at infinity, as treated in \cite{MR1415232,MR1681098,MR2073330,MR2184099,Cap}. More precisely, suppose that $r(x,k)=r(k)$ for all $x\in V$ and $k\ge 1$, where the function $r\colon \N\to(0,\infty)$ satisfies 
\begin{eqnarray}
\label{lip1}
\sup_{k\ge 1}|r(k+1)-r(k)| & < & \infty,\\
\label{lip2}
\inf_{k-\ell\ge \delta}r(k)-r(\ell) & > & 0,
\end{eqnarray}
for some $\delta\in\N$. Under these conditions, it was shown in \cite{MR2073330} that  
\begin{eqnarray}
\lambda\left(\textsc{zrp}(K_n,{r(\cdot)},m)\right) & \ge & c,
\end{eqnarray}
for some $c>0$ that does not depend on $n,m$. We immediately deduce the following.
\begin{corollary}[Homogeneous Lipschitz rates increasing at infinity]\label{co:lip} Under assumptions (\ref{lip1})-(\ref{lip2}),  for any doubly stochastic matrix $P$,
\begin{eqnarray}
\lambda\left(\textsc{zrp}(P,{r(\cdot)},m)\right) & \ge & c\lambda(P),
\end{eqnarray}
for some constant $c>0$ which depends neither on $P$, nor  on $m$.
\end{corollary}

As a final example, consider the rate-one case ${\bf r}={\bf 1}$ (i.e.\ $r(x,k)=1$ for all $x\in V$ and all $k\ge 1$). For this natural and important choice, Morris \cite{MorrisZRP} showed that
\begin{eqnarray}
\lambda\left(\textsc{zrp}(K_n,{\mathbf{1}},m)\right) & \ge & {c}\left(1+ \frac{m}{n} \right)^{-2},
\end{eqnarray}
for some universal constant $c>0$. By virtue of Corollary \ref{co:mf}, this immediately yields a lower-bound on $\lambda\left(\textsc{zrp}(P,{\mathbf{1}},m)\right)$. Below, we will complement it with a matching upper-bound to obtain the following general result, which completely settles the rate-one case.
 \begin{corollary}[Unit rates]\label{co:unit} For any $n\times n$ doubly stochastic matrix $P$, we have
\begin{eqnarray}
\lambda\left(\textsc{zrp}(P,{\mathbf{1}},m)\right) & \asymp & \lambda(P)\left(1+\frac{m}{n}\right)^{-2},
\end{eqnarray}
where $\asymp$ denotes equality up to universal constants.
\end{corollary}
This applies in particular to the case where $P$ is the transition matrix of simple random walk on a regular graph. To the best of our knowledge, the only  graphs for which the answer was known were the complete graph and the torus  \cite{MorrisZRP}. Interestingly, the work \cite{MorrisZRP} also relies on a (weaker) comparison to the mean-field setting, see Remark \ref{congestion} below.

As a final application, let us turn our attention to the total-variation and $L_{\infty}-$ mixing times, defined respectively by
\begin{eqnarray}
\mix^{(\textsc{tv)}} & := & \inf \left\{t\ge 0\colon \max_{\xi\in\Omega} \sum_{\eta\in\Omega} \left|p_t(\xi,\eta)-\mu(\eta)\right|\le \frac 1e \right\};\\
\mix^{(\infty)} & := & \inf \left\{t\ge 0\colon \max_{(\xi,\eta)\in\Omega^2}\left|\frac{p_t(\xi,\eta)}{\mu(\eta)} - 1\right| \le \frac 1e \right\},
 \end{eqnarray} 
 where $p_t=e^{t\cL}$ denotes the transition kernel of the chain. 
Estimating these fundamental parameters is in general a challenging task, see the books \cite{levin,MR2341319}. To the best of our knowledge, the mixing-time of the rate-one $\mathbf{ZRP}$ has only been determined on the cycle \cite{Lacoin} (via a bijection with the exclusion process, specific to the cycle) and, very recently, on the complete graph \cite{MS}. Both results concern the classical regime where the total density of particles per site remains bounded. In this setting, Corollary \ref{co:unit} is actually powerful enough to provide the exact order of magnitude of the mixing times on all expanders. 
\begin{corollary}[Mixing times]For each $n\ge 1$, consider an $n\times n$ bi-stochastic matrix $P_n$ and an integer $m_n\ge 1$. Assume that $\lambda(P_n) =  \Omega(1)$ and that $m_n = \Theta(n)$ as $n\to\infty$. Then,
\begin{eqnarray}
\mix^{(\textsc{tv)}}\left(P_n,{\mathbf 1},m_n\right) \ = \ \Theta(n) & \textrm{ and } &
\mix^{(\infty)}\left(P_n,{\mathbf 1},m_n\right) \ = \ \Theta(n).
\end{eqnarray}
\end{corollary}
Indeed, for the upper-bound,  we may exploit the well-known estimate (see, e.g., \cite{MR2341319})
\begin{eqnarray}
\label{e:mixrel}
  \mix^{(\infty)} & \le & \frac{2}{\lambda}  \log \frac{e}{\min_{\eta}\mu(\eta)}.
\end{eqnarray}
Under our assumptions, Corollary \ref{co:unit} guarantees that the right-hand side is $\Theta(n)$. For a matching lower-bound on $\mix^{(\textsc{tv})}$, simply observe that if all particles start on the same site $x$, it will takes time $\Omega(m)$ for, say, half of the particles to even leave $x$.  

We conclude this section with two important remarks, and a question.

\begin{remark}[Congestion]\label{congestion}The support $E=\{(x,y)\in V\times V\colon P(x,y)>0\}$ of the matrix $P$ naturally defines a directed graph on $V$. Suppose that on this graph, an arbitrary path $\gamma_{x\to y}$ is specified between each pair  $(x,y)$ of sites, and define the resulting congestion as
\begin{eqnarray}
\kappa  & := & \max_{(a,b)\in E}\frac{1}{\pi(a)P(a,b)}\sum_{(x,y)\in S^2}\pi(x)\pi(y)\,\mathrm{length}(\gamma_{x\to y})\,{\bf 1}_{(\gamma_{x\to y}\textrm{ traverses }(a,b))}.
\end{eqnarray}
Minimizing this quantity over all possible paths (or more generally: convex combinations of such paths) defines what is known as the \emph{congestion constant} $\kappa(P)$ of the matrix $P$. Its appeal lies in the following inequality, discovered by  Diaconis and Stroock \cite{DiaconisStroock}, who further developed earlier results of Jerrum and Sinclair \cite{JS} along these lines:  
\begin{eqnarray}
\lambda(P) & \ge & \frac{1}{\kappa(P)}.
\end{eqnarray}
Now, when $P$ is the transition matrix of a graph $G$, it is often the case that the congestion constant can be \emph{lifted} to certain particle systems on $G$. This method is due to Diaconis and Saloff-Coste  who first applied it to the exclusion process \cite{comparison}.  For the rate-one $\textbf{ZRP}$ this was observed by Morris \cite{MorrisZRP}, who  obtained a (restricted) version of Corollary \ref{co:mf} with our optimal factor $\lambda(P)$ replaced by the weaker $\frac{1}{\kappa(P)}$. In the case where $G$ is the torus, those two factors are of the same order, yielding the correct order of magnitude for $\lambda\left(\textsc{zrp}(P,\textbf{1},m)\right)$. In general however, the quantities $\lambda(P)$ and $\frac{1}{\kappa(P)}$ will differ significantly, making the sharp estimate provided in Corollary \ref{co:unit} completely out of reach of a congestion-based comparison.
\end{remark}

\begin{remark}[Beyond the Poincar\'e constant] All consequences listed above rely on the Poincar\'e constant only. However, our main result actually provides a control of the Dirichlet form itself, which carries much more information. In particular, Corollary \ref{co:mf} can also be used to transfer any known result on the logarithmic Sobolev constant (such as the estimate obtained, e.g.,  in \cite{MR2184099}), the whole spectrum of $\cL$, or the average $L^2$ distance to equilibrium. We therefore expect to see many other applications of our comparison result in the future.
\end{remark}
As already mentioned, our results bare resemblance to Aldous' spectral-gap conjecture, established by Caputo, Liggett and Richthammer \cite{Caputo}, which reduces the Poincar\'e constant of the interchange process to that of the random walk performed by a single particle. Since our comparison holds at the stronger level of the Dirichlet form itself, we may ask the following question, in relation with a recent work of Alon and Kozma \cite{Alon} (which provides such a comparison result, but with $\la(P)$ replaced with the inverse of the total-variation mixing time corresponding to $P$, and involves another term which is often of order 1). 
\begin{question}
Does an analog of Theorem \ref{thm:main} hold for the exclusion/interchange processes?
More precisely, can one relate the Dirichlet forms of two exclusion/interchange processes (with the same number of sites), via the Dirichlet forms corresponding to a single particle in these processes? Note that this makes sense for arbitrary edge weights.     
\end{question}
As the interchange process is a transitive chain, an affirmative answer would imply (see, e.g.\ \cite[Corollary 8.8]{aldous}) an upper bound on the $L_{\infty}$ mixing time of the interchange process on any $n$-vertex regular graph, in terms of  the inverse  spectral gap of a single particle, times the $L_{\infty}$ mixing time of the interchange process on the complete graph on $n$ vertices (which is known to be of order $\log n$, see \cite{DS}).

\section{Proofs}
\subsection{Main comparison theorem}
We will need to consider configurations with one particle removed, so let us introduce
\begin{eqnarray}
\widehat{\Omega} & := & \left\{\xi\in\Z_+^V\colon \sum_{x\in V}\xi(x)=m-1\right\}.
\end{eqnarray}
We extend the definition of $\mu$ to $\widehat{\Omega}$ by using the same product formula (the right-hand side of (\ref{statio}) makes perfect sense for any configuration $\eta\in\Z_+^V$). Note that we do not modify the normalization constant $Z$, so only the restriction of $\mu$ to $\Omega$ is guaranteed to be a probability measure. In view of the product form of $\mu$, we then have
\begin{eqnarray}
\label{prod}
\mu(\xi+\delta_x)r(x,\xi(x)+1) & = &  \mu(\xi)\pi(x),
\end{eqnarray}
for any $\xi\in\widehat{\Omega}$ and any $x\in V$. Given an observable $f\colon \Omega\to\R$ and a configuration $\xi\in\widehat{\Omega}$, we define a function $f_\xi\colon V\to \R$ by the formula
\begin{eqnarray}
\label{proj}
f_\xi\left(x\right) &:= & f\left(\xi+\delta_x\right).
\end{eqnarray}
With this notation in hands, we may now state our main identity, which provides a simple  connection between the Dirichlet form of the \textbf{ZRP} and that of the jump matrix. 
\begin{lemma}[Main identity]\label{lm:main}For any $f,g\colon \Omega\to\R$, we have
\begin{eqnarray}
\cE_{\textsc{zrp}(P,{\bf r},m)}\left(f,g\right) & = & \sum_{\xi\in\widehat{\Omega}}\mu(\xi)\cE_P\left(f_\xi,g_\xi\right).
\end{eqnarray}
\end{lemma}
\begin{proof}
Explicitating the definition of $\cE_{\textsc{zrp}(P,{\bf r},m)}$, we have
\begin{eqnarray*}
\cE_{\textsc{zrp}(P,{\bf r},m)}\left(f,g\right) & = & \sum_{\eta\in \Omega}\sum_{(x,y)\in V^2}\mu(\eta)r(x,\eta(x))P(x,y)f(\eta)\left(g(\eta)-g(\eta+\delta_y-\delta_x)\right).
\end{eqnarray*}
Since $r(x,0)=0$ for all $x\in V$, only the triples $(\eta,x,y)$ with $\eta(x)\ge 1$ contribute to this sum. Consequently, we may use the change of variables $\xi=\eta-\delta_x$ to rewrite this as
\begin{eqnarray*}
\cE_{\textsc{zrp}(P,{\bf r},m)}\left(f,g\right) & = & \sum_{\xi\in\widehat{\Omega}}\sum_{(x,y)\in V^2} \mu(\xi+\delta_x)r(x,\xi(x)+1)P(x,y)f_\xi(x)\left(g_\xi(x)-g_\xi(y)\right)\\
& = & \sum_{\xi\in\widehat{\Omega}}\sum_{(x,y)\in V^2}\mu(\xi)\pi(x)P(x,y)f_\xi(x)\left(g_\xi(x)-g_\xi(y)\right)\\
& = & \sum_{\xi\in\widehat{\Omega}}\mu(\xi)\cE_P\left(f_\xi,g_\xi\right),
\end{eqnarray*}
where we have successively used our definition (\ref{proj}), the relation (\ref{prod}), and the definition of the Dirichlet form $\cE_P$. 
\end{proof}
To prove the sharpness of our comparison, we will also need to \emph{lift} observables from $V$ to $\Omega$ in a suitable way.  This is provided by the following elementary lemma.
\begin{lemma}[Useful lift]\label{lift}Let $\phi\colon V\to\R$ be arbitrary, and define an observable $f$ on $\Omega$ by
\begin{eqnarray}
f(\eta) & := & \sum_{x\in V}\phi(x)\eta(x).
\end{eqnarray} 
We then have the identity
\begin{eqnarray}
\cE_{\textsc{zrp}(P,{\bf r},m)}(f,f) & = & \mu(\widehat{\Omega})\,\cE_P\left(\phi,\phi\right).
\end{eqnarray}
\end{lemma}
\begin{proof}
For any $\xi\in\widehat{\Omega}$ and any $x\in V$, we have by construction
\begin{eqnarray*}
f_\xi(x) & = & f(\xi)+\phi(x),
\end{eqnarray*}
where we have extended the definition of $f$ to $\widehat{\Omega}$ in the obvious way. Since shifting an observable $h$ by a constant does not affect $\cE_P(h,h)$, we deduce that
\begin{eqnarray*}
 \cE_P(f_\xi,f_\xi) & = & \cE_P(\phi,\phi).
 \end{eqnarray*} The conclusion now follows from Lemma \ref{lm:main}. 
\end{proof}
Those two lemmas are already enough to establish our comparison theorem.
\begin{proof}[Proof of Theorem \ref{thm:main}]Fix $\kappa>0$, and suppose first that (\ref{comp:jump}) holds. In particular, for any observable $f\colon \Omega\to\R$ and any configuration $\xi\in\Omega$, we have 
\begin{eqnarray*}
\cE_{P}(f_\xi,f_\xi) & \ge & \kappa\, {\cE}_{Q}(f_\xi,f_\xi).
\end{eqnarray*}
Multiplying by $\mu(\xi)$ and summing over all $\xi\in\widehat{\Omega}$ yields
\begin{eqnarray*}
\sum_{\xi\in\widehat{\Omega}}\mu(\xi)\cE_{P}(f_\xi,f_\xi) & \ge & \kappa\, \sum_{\xi\in\widehat{\Omega}}\mu(\xi){\cE}_{Q}(f_\xi,f_\xi).
\end{eqnarray*}
By  Lemma \ref{lm:main}, this gives precisely (\ref{comp:zrp}). Conversely, let us now suppose that (\ref{comp:jump}) fails. This means that there is a function  $\phi\colon V\to\R$ such that 
\begin{eqnarray*}
\cE_{P}(\phi,\phi) & < & \kappa\, {\cE}_{Q}(\phi,\phi).
\end{eqnarray*}
The construction in Lemma \ref{lift} provides us with an observable $f$ on $\Omega$ such that
\begin{eqnarray*}
\cE_{\textsc{zrp}(P,{\bf r},m)}(f,f) &  = & \mu(\widehat{\Omega})\,\cE_P\left(\phi,\phi\right);\\
\cE_{\textsc{zrp}(Q,{\bf r},m)}(f,f) &  = & \mu(\widehat{\Omega})\,\cE_Q\left(\phi,\phi\right).
\end{eqnarray*}
It follows that $\cE_{\textsc{zrp}(P,{\bf r},m)}(f,f)<\kappa \cE_{\textsc{zrp}(Q,{\bf r},m)}(f,f)$, and (\ref{comp:zrp}) fails as desired.
\end{proof}
All our announced lower-bounds on the Poincar\'e constant follow from this general theorem. To complete the proof of Corollary \ref{co:unit}, we need to provide a matching upper-bound. 
\subsection{Matching upper-bounds on the Poincar\'e constant}

\begin{lemma}[A general upper-bound]\label{ub} Assume that $P$ is doubly stochastic and that ${\mathbf r}$ is homogeneous (i.e.\ $r(x,k)=r(k)$ for some fixed $r\colon \N\to(0,\infty)$). Then,
\begin{eqnarray}
\lambda\left(P,\textbf{r},m\right) & \le & \left(1-\frac 1n\right)\lambda(P)\frac{\E_{\mu}[r(\zeta)]}{\mathrm{Var}_{\mu}(\zeta)},
\end{eqnarray}
where $\zeta$ is the number of particles on an arbitrarily fixed site, under the stationary law $\mu$. 
\end{lemma}
\begin{proof}
Let $n=|V|$. Since $\pi$ is the uniform law and the rates are homogeneous, the law of $\left(\eta(x)\colon x\in V\right)$ is exchangeable under the stationary law $\mu$. In particular, we may write
\begin{eqnarray*}
\mathrm{Cov}\left(\eta(x),\eta(y)\right) & = & \alpha{\bf 1}_{(x=y)}-\beta,
\end{eqnarray*}
for certain $\alpha,\beta\in\R$. Summing over all $(x,y)\in V^2$, we see that
\begin{eqnarray*}
n\alpha - \beta n^2 & = & \textrm{Var}\left(\sum_{x\in V}\eta(x)\right) \ = \ 0,
\end{eqnarray*}
because the total number of particles is $m$. Thus, we actually have $\beta=\frac{\alpha}{n}$ and hence 
\begin{eqnarray*}
\alpha & = & \frac{n}{n-1}\textrm{Var}(\eta(x)),
\end{eqnarray*}
where $x$ is an arbitrary site. Now, let $\phi\colon V\to\R$ realize the minimum in the definition of $\lambda(P)$. Upon  recentering and rescaling, we may actually assume that
\begin{eqnarray*}
\sum_{x\in V}\phi(x) & = & 0\\
\sum_{x\in V}\left(\phi(x)\right)^2 & = & n\\
\cE_P\left(\phi,\phi\right) & = &  \lambda(P).
\end{eqnarray*} 
Applying the construction in Lemma \ref{lift}, we obtain an observable $f$ on $\Omega$ such that
\begin{eqnarray*}
\cE_{\textsc{zrp}(P,{\bf r},m)}(f,f) &  = & \lambda(P)\mu(\widehat{\Omega}).
\end{eqnarray*} 
Moreover, explicitating the definition of $f$, we see that
\begin{eqnarray*}
\textrm{Var}(f) & = & \sum_{(x,y)\in V^2}\phi(x)\phi(y)\mathrm{Cov}\left(\eta(x),\eta(y)\right)\\
& = & \alpha\sum_{x\in V^2}\left(\phi(x)\right)^2-\beta\left(\sum_{x\in V}\phi(x)\right)^2\\
& = & \frac{n^2}{n-1}\textrm{Var}(\eta(x)),
\end{eqnarray*}
Finally, note that by summing (\ref{prod}) over all $(\xi,x)\in\widehat{\Omega}\times V$ and then performing the change of variables $\eta=\xi+\delta_x$, we obtain
\begin{eqnarray*}
\mu(\widehat{\Omega}) 
& = & \sum_{(\eta,x)\in\Omega\times V}\mu(\eta)r(x,\eta(x)),
\end{eqnarray*}
which equals  $n\E[r(\eta(x))]$ by exchangeability. In conclusion, 
\begin{eqnarray*}
\frac{\cE_{\textsc{zrp}(P,{\bf r},m)}(f,f)}{\textrm{Var}(f)} & = & \left(1-\frac{1}{n}\right)\frac{\lambda(P)\E\left[r(\eta(x))\right]}{\textrm{Var}(\eta(x))},
\end{eqnarray*}
which completes the proof, by definition of the Poincar\'e constant $\lambda\left({\textsc{zrp}(P,{\bf r},m)}\right)$.
\end{proof}
Let us now specialize this to the rate-one case to complete the proof of Corollary \ref{co:unit}. 
\begin{proof}[Proof of the upper-bound in Corollary \ref{co:unit}]
The stationary law $\mu$ is simply the uniform law on the set $\Omega$, which has cardinality ${n+m-1\choose m}$. Thus, the law of $\zeta$ is explicitly given by  
\begin{eqnarray*}
\P\left(\zeta=k\right) & = & \frac{{n+m-k-2\choose n-2}}{{n+m-1\choose n-1}} \qquad (0\le k\le m).
\end{eqnarray*}
Let us compute $\rm{Var}(\zeta)$. Using standard binomial manipulations, it is easy to check that
\begin{eqnarray*}
\sum_{k=0}^mk(k-1){n+k-2\choose n-2} & = & n(n-1){n+m-1 \choose n+1}.
\end{eqnarray*}
Dividing through by ${n+m-1\choose n-1}$, we obtain
\begin{eqnarray*}
\E\left[(m-\zeta)(m-\zeta-1)\right]  & = & \frac{n-1}{n+1}m(m-1).
\end{eqnarray*}
On the other hand, without having to do any computation, we know that 
$
\E[\zeta] = \frac{m}{n}
$ by exchangeability of $\mu$. Together, these two facts easily lead to
\begin{eqnarray*}
\textrm{Var}(\zeta) & = & \frac{n-1}{n+1}\left(1+\frac mn\right)\frac{m}{n}.
\end{eqnarray*}
Finally, since $r(k)={\bf 1}_{(k\ge 1)}$, we have
\begin{eqnarray*}
\E[r(\zeta)] & = & \frac{m}{n+m-1}.
\end{eqnarray*}
By Lemma \ref{ub}, we conclude that
\begin{eqnarray*}
\lambda\left(P,\textbf{r},m\right) & \le & \frac{\lambda(P)n(n+1)}{(n+m)(n+m-1)} \ \asymp \ \lambda(P)\left(1+\frac{m}{n}\right)^{-2},
\end{eqnarray*}
as desired.
\end{proof}


\nocite{}
\bibliographystyle{plain}
\bibliography{ZRP}

\vspace{2mm}

\end{document}